\documentclass[12pt, letterpaper]{article}
\usepackage{amsfonts}
\usepackage{amssymb}
\usepackage{latexsym}  
\usepackage{pifont}  
\usepackage{bbm}  %

\let\SavedRightarrow=\Rightarrow
\usepackage{marvosym}
\let\Rightarrow=\SavedRightarrow

\usepackage{amsmath}   

\newenvironment{itemizz}{\begin{itemize}\setlength{\itemsep}{-1mm}}
{\end{itemize}}

\newenvironment{itemizn}[1] 
{\begin{itemize} \setlength{\itemsep}{-1mm} %
\renewcommand\labelitemi{\ding{#1}}} %
{\end{itemize}}

\newtheorem{theorem}{Theorem}[section]
\newtheorem{definition}[theorem]{Definition}
\newtheorem{lemma}[theorem]{Lemma}

\newcommand\NNN{\mathbb {N}}
\newcommand\RRR{\mathbb {R}}
\newcommand\PPP{\mathbb {P}}

\newcommand\BB{\mathcal {B}}
\newcommand\JJ{\mathcal {J}}
\newcommand\TT{\mathcal {T}}

\newcommand\dom{\mathrm{dom}}  
\newcommand\cl{\mathrm{cl}}  
\newcommand\lh{\mathrm{lh}}   
\newcommand\MA{\mathrm{MA}}  
\newcommand\PFA{\mathrm{PFA}}  
\newcommand\SOCA{\mathrm{SOCA}}  
\newcommand\ZFC{\mathrm{ZFC}}  
\newcommand\OCAARS{\mathrm{OCA}_{\mathrm{[ARS]}}}  
\newcommand\CH{\mathrm{CH}}  
\newcommand\diam{\mathrm{diam}}  
\newcommand\hgt{\mathrm{ht}}   

\newcommand\res{\mathord {\upharpoonright}}  

\newcommand\cat{^{\mathord{\frown}}}  

\newcommand\onto{\twoheadrightarrow}

\newcommand\iv{^{-1}} 

\newcommand\one{\mathbbm{1}} 

\newcommand\eop{{\Large \Coffeecup}}  

\newenvironment{proof}{{\bf Proof.}}{\eop\medskip}
\newenvironment{proofof}[1]{\medskip \textbf{Proof of #1.}}{\eop\medskip}

\begin{document}

\title{Arcs in the Plane%
}

\author{Joan E. Hart\footnote{University of Wisconsin, Oshkosh,
WI 54901, U.S.A.,
\ \ hartj@uwosh.edu}
\  and
Kenneth Kunen\footnote{University of Wisconsin,  Madison, WI  53706, U.S.A.,
\ \ kunen@math.wisc.edu}
\thanks{Both authors partially supported by NSF Grant 
DMS-0456653.}
}

\maketitle

\begin{abstract}
Assuming PFA, every uncountable subset $E$ of the plane meets
some $C^1$ arc in an uncountable set.
This is not provable from $\MA(\aleph_1)$, although in the case
that $E$ is analytic, this is a ZFC result.
The result is false in ZFC for $C^2$ arcs,
and the counter-example is a perfect set.
\end{abstract}

\section{Introduction}
\label{sec-intro}
As usual, an \emph{arc} in $\RRR^n$ is a set homeomorphic to a closed
bounded subinterval of  $\RRR$.
A (simple) \emph{path} is a homeomorphism $g$ mapping a compact
interval onto $A$.
For $k \ge 1$,
a path is $C^k$ iff it is a $C^k$ function, and
an arc $A$ is $C^k$  iff 
$A$ is the image of some $C^k$ path $g$, with
$g'(t) \ne 0$ for all $t$; 
equivalently, $A$ has a $C^k$ arc length parameterization.
Also, $A$ is $C^{\infty}$ iff it is $C^k$ for all $k$.
We consider the following:

\smallskip

\textit{\textbf{Question.}} For $n \ge 2$, 
if $E \subseteq \RRR^n$ is uncountable,
must there be a ``nice'' arc $A$ such that $E \cap A$ is uncountable?

\smallskip

Obviously, the answer will depend on the definition of ``nice''.
We should expect
ZFC results for closed $E$ (equivalently, for analytic $E$),
and independence results for arbitrary $E$.
In general, under $\CH$ things are as bad as possible, and
under $\PFA$, things are as good as possible.
In most cases, the results are the same for all $n \ge 2$, and trivial
for $n = 1$.

For arbitrary arcs, the results are quite old.
In ZFC, every closed uncountable set meets some
arc in an uncountable set.
For $n \ge 2$, arcs are nowhere dense in $\RRR^n$; so 
under $\CH$ there is a Luzin set that
meets every arc in a countable set. 
At the other extreme,
under $\MA(\aleph_1)$, every uncountable $E \subseteq \RRR^n$
meets some arc in an uncountable set.  

If ``nice'' means ``straight line'', then there is a trivial
counter-example: a perfect set $E$ which meets every 
line in at most two points.

Paper \cite{Ku} introduces results where ``nice'' means
``almost straight'':

\begin{definition}
\label{def-directed}
Let $\rho : \RRR^n \backslash \{0\} \onto S^{n-1}$
be the perpendicular retraction given by $\rho(x) = x / \|x\|$. 
Then $A \subseteq \RRR^n$ is $\varepsilon$--\emph{directed}
iff for some $v \in S^{n-1}$,
$\| \rho(x-y) - v \| \le \varepsilon$ or
$\| \rho(x-y) + v \| \le \varepsilon$
whenever $x,y$ are distinct points of $A$.
\end{definition}

The retraction $\rho(x-y)$ may be viewed as the \emph{direction}
from $y$ to $x$.
Every $A\subseteq \RRR^n$ is trivially $\sqrt2$--directed, 
and $A$ is $0$--directed iff $A$ is contained in a straight line.
If ``nice'' means ``$\varepsilon$--directed'',
a counter-example to the Question
is consistent with $\MA(\aleph_1)$. By 
\cite{Ku}, 
the existence of a \emph{weakly} Luzin set is consistent 
with $\MA(\aleph_1)$, and 
whenever $\varepsilon < \sqrt2$,
a {weakly} Luzin set (see \cite{Ku} Definition 2.4)  
meets every $\varepsilon$--directed set in a countable set.
However, under SOCA, which follows from PFA,
whenever $\varepsilon > 0$, every 
uncountable set meets some  $\varepsilon$--directed arc in an uncountable set
(see Lemma \ref{lemma-dir}).
Every $C^1$ arc is a finite union of $\varepsilon$--directed arcs,
and hence we get the stronger:

\begin{theorem}
\label{thm-c1-pfa}
PFA implies that every uncountable subset of $\RRR^n$
meets some $C^1$  arc in an uncountable set.
\end{theorem}
$\MA(\aleph_1)$ is not sufficient for this theorem,
because, as in the $\varepsilon$-directed case ($\varepsilon <\sqrt{2})$,
a weakly Luzin set provides a counter-example.
Theorem \ref{thm-c1-pfa} and 
the following ZFC theorem for closed sets 
are proved in Section \ref{sec-posit}.

\begin{theorem}
\label{thm-c1}
If $P \subseteq \RRR^n$ is closed and uncountable,
then there is a $C^1$ arc $A$ with a Cantor set $Q \subseteq P \cap A$.
Hence, for every $\varepsilon > 0$, $P$ meets 
some $\varepsilon$--directed arc in an uncountable set.
\end{theorem}

If the Question asks for a $C^2$ arc, 
then a ZFC counter-example exists
in the plane, and hence in any $\RRR^n$ ($n \ge 2$).
The counter-example, given in Theorem \ref{thm-c2}, is
a \textit{non-squiggly} subset of the plane.
A simple example of a non-squiggly set is a $C^1$ arc whose tangent vector
either always rotates clockwise or always rotates counter-clockwise.
In particular,  such an arc may be
the graph of a convex function $f \in C^1([0,1],\RRR)$; 
a real differentiable function is \textit{convex} iff its derivative
is a monotonically increasing function.
But non-squiggly makes sense for non-smooth arcs,
and in fact for arbitrary subsets of the plane:

\begin{definition}
\label{def-squig}
$A \subseteq \RRR^2$ is \emph{non-squiggly} iff
there is a $\delta$, with $0<\delta \le \infty$, such that  whenever
$\{x,y,z,t\} \in [A]^4$ and
$diam(\{x,y,z,t\}) \le \delta$, point $t$ 
is not interior to triangle $xyz$.
\end{definition}

\begin{theorem}
\label{thm-c2}
There is a perfect non-squiggly set $P \subseteq \RRR^2$
which lies in a $C^1$ arc $A$
and which meets each $C^2$ arc in a finite set.
Moreover, the $C^1$ arc $A$ may be taken to be
the graph of a convex function.
\end{theorem}

As ``nice'' notions, non-squiggly is orthogonal to smooth:

\begin{theorem}
\label{thm-squig}
There is a perfect set $P \subseteq \RRR^2$ which lies in a 
$C^\infty$ arc and which meets every non-squiggly set in a countable set.
\end{theorem}

Note that by Ramsey's Theorem, every infinite set in $\RRR^2$
has an infinite non-squiggly subset.

In Definition \ref{def-squig}, allowing $\delta < \infty$ makes
non-squiggly a local notion; so, piecewise linear arcs and
some spirals (such as $r = \theta \; ;\;  0 \le \theta < \infty$) 
are non-squiggly.
However, the results of this paper would be unchanged if we 
simply required $\delta=\infty$.
For $0 < \delta \le \infty$, if $E\subseteq \RRR^2$ 
meets a non-squiggly set $A$ in an uncountable set, then  
$E$ has uncountable intersection
with a subset of $A$ whose diameter is at most $\delta$.

The proof of Theorem \ref{thm-c2} uses the assumption
that each $C^2$ arc is parameterized by some $g$
whose derivative is nowhere $0$.
Dropping this requirement on $g'$ yields a weaker
notion of $C^\infty$, and a different result.
Call a $C^k$ arc \textit{strongly} $C^k$, and
say that an arc is \emph{weakly} $C^k$ iff it is the image of 
a $C^k$ path.
Then, an arc is weakly $C^{\infty}$ iff it is weakly $C^k$ for all $k$.

\begin{theorem}
\label{thm-meet}
If $E \subseteq \RRR^n$ is bounded and infinite, then 
it meets some weakly $C^\infty$ arc in an infinite set.
\end{theorem}

Theorems \ref{thm-c2} and \ref{thm-squig} are proved 
in  Section \ref{sec-neg}; 
Theorem \ref{thm-meet} and some related facts are proved in 
Section \ref{sec-rmks}.

\section{Remarks on Hermite Splines}
\label{sec-hermite}
We construct the arc of Theorem \ref{thm-c1}
by first  producing a ``nice'' Cantor set $Q \subseteq P$. 
Then we apply results,
described in this section,
that make it possible to draw a smooth curve through a closed set.
These results are a natural extension of results of Hermite
for drawing a curve through a finite set.
Our proof of Theorem \ref{thm-c1}
reduces the problem to the case where $Q \subset \RRR^2$
is the graph of a function with domain $D \subset \RRR$; 
then we extend this function to all of $\RRR$ to produce 
the desired arc.

First consider the case $|D| = 2$, or 
interpolation on an interval $[a_1, a_2]$; we find
$f \in C^1(\RRR)$ with predetermined values $b_1,b_2$
and slopes $s_1,s_2$ at $a_1,a_2$, \emph{and} we bound $f,f'$
on $[a_1, a_2]$ in terms of the \emph{three} slopes:
$s := (b_2 - b_1)/(a_2 - a_1)$, and $s_1, s_2$.
Following Hermite, $f$ will be the natural
cubic interpolation function.  Our bounds show that if
$s, s_1, s_2$ are all close to each other, then $f$ is close to
the linear interpolation function $L$.

\begin{lemma}
\label{lemma-interp-cubic}
Given $s_1,s_2,b_1,b_2$ and $a_1 < a_2$, let
$s = (b_2 - b_1)/(a_2 - a_1)$, and let
$L(x) = b_1 + s(x - a_1)$.  Let
$M = \max(|s_1 - s|,|s_2 - s|)$. 
Then there is a cubic $f$
with each $f(a_i) = b_i$ and each $f'(a_i) = s_i$,
such that 
\begin{itemizz}
\item[1.] $|(f(x_2) - f(x_1)) / (x_2 - x_1) - s| \le 3M$ whenever
$a_1 \le x_1 < x_2 \le a_2$.
\end{itemizz}
Moreover, for all $x \in [a_1,a_2]$:
\begin{itemizz}
\item[2.] $|f'(x) - s| \le 3M$.
\item[3.] $|f(x) - L(x)| \le 2M(a_2 - a_1)$.
\end{itemizz}
\end{lemma}
\begin{proof}
(1) follows from (2) and the Mean Value Theorem.   Now, let
\begin{align*}
f(x) &=
L(x) + \beta_2 (x - a_1)^2 (x - a_2) + \beta_1 (x - a_1) (x - a_2)^2 \\
f'(x) &=
s + \beta_2 (x - a_1)^2 + \beta_1  (x - a_2)^2 +
2 (\beta_2 + \beta_1) (x - a_1) (x - a_2) \ \ .
\end{align*}
Then $f(a_i) = b_i$ is obvious, and
setting $\beta_i = (s_i - s)/ (a_2 - a_1)^2$ we get $f'(a_i) = s_i$.
To see (2) and (3), note that
 $|\beta_i| \le  M / (a_2 - a_1)^2$,
and $(x-a_1)(a_2-x)\le (a_2-a_1)^2/4$ 
(the maximum of $(x-a_1)(a_2-x)$ occurs
at the midpoint $x=\frac{a_1+a_2}{2}$). 
\end{proof}

Next, we consider extending, to all of $\RRR$,
a $C^1$ function defined on a closed $D \subset \RRR$.  First
note that there are two possible
meanings for ``$f \in C^1(D)$'':

\begin{definition}
Assume that $f,h \in C(D,\RRR)$, where $D$ is a closed subset of $\RRR$.
Then $f' = h$ \emph{in the strong sense} iff 
\[
\begin{array}{ll}
& \forall x \in D \; \forall \varepsilon > 0 \;
\exists \delta > 0  \;
\forall x_1,x_2 \in D \\
&\qquad \left[ x_1 \ne x_2 \ \&\ |x_1 - x|, |x_2 - x| < \delta
\ \longrightarrow\  \left|\frac{f(x_2) - f(x_1)}{x_2 - x_1} - h(x) \right| < 
        \varepsilon \right]\ \ .
\end{array} 
\]
\end{definition}
The usual or weak sense would only require this with $x_1$ replaced
by the point $x$.  When $D$ is an interval, the two senses are equivalent
by the continuity of $h$ and the Mean Value Theorem.
Note that $f' = h$ in the strong sense iff
there is a $g \in C(D\times D, \RRR)$ such that $g(x,x) = h(x)$ for
each $x$ and
$g(x_1,x_2) = g(x_2, x_1) =
(f(x_2) - f(x_1)) / (x_2 - x_1)$ whenever $x_1 \ne x_2$.

If $D$ is finite, then  $f' = h$ in the strong sense for \emph{any}
$f,h : D \to \RRR$, and the cubic Hermite spline is 
an  $\tilde f \in C^1(\RRR,\RRR)$ with
$\tilde f \res D = f$ and $\tilde f' \res D = h$.
The following lemma generalizes this to an arbitrary closed $D$:

\begin{lemma}
\label{lemma-interp}
Assume that $f,h \in C(D,\RRR)$, where $D$ is a closed subset of $\RRR$,
and $f' = h$ in the strong sense.  Then there
are $\tilde f, \tilde h \in C(\RRR,\RRR)$ such that
$\tilde f' =  \tilde h $,
$\tilde f \supseteq f$, and
$\tilde h \supseteq h$.
\end{lemma}
\begin{proof}
Let $\JJ$ be the collection of pairwise disjoint open intervals covering 
$\RRR \backslash D$. 
For each interval $J\in \JJ$, we shall define 
$\tilde f, \tilde h$ on $J$.

If $J$ is the unbounded interval $(a_1, \infty)$, with $a_1 \in D$,
define $\tilde f$ and $\tilde h$ by the linear
$\tilde f(x) = f(a_1) + (x  - a_1) h(a_1)$ and 
$\tilde h(x) = h(a_1)$, for $x \in J$. 
Then $\tilde f, \tilde h$ are continuous on $\overline J$
and $\tilde f' =  \tilde h $ on $J$.  
At $a_1$,
the derivative of $\tilde f$ from the right is $h(a_1)$;
the derivative of $\tilde f$ from the left,
as well as the continuity of $\tilde f, \tilde h$ from the left,
depend on how we extend $f$ to the bounded intervals.

The unbounded interval $(-\infty, a_2)$ is handled likewise.

Say $J = (a_1, a_2)$, with $a_1,a_2 \in D$.
On $J$, let $\tilde f$ be the cubic obtained from Lemma
\ref{lemma-interp-cubic}, with $b_i = f(a_i)$ and $s_i = h(a_i)$.
Then $\tilde h$ is the quadratic $\tilde f '$ on $J$.

To finish, we verify that
$\tilde f, \tilde h$ are continuous and $\tilde f' =  \tilde h $
on $\RRR$.
Fix $z \in D$.  
Since differentiability implies continuity,
it suffices to show that $\tilde h$ is continuous at $z$, and that
$h(z)=\tilde f'(z) = 
 \lim_{x \to z} (\tilde f(x) - \tilde f(z)) / (x - z)$.
We  verify 
the continuity of $\tilde h$ from the left at $z$, and 
the difference quotient's limit for $x$
approaching $z$ from the left; 
a similar argument handles these from the right.  
Let $\sigma =  h(z) = \tilde h(z)$.    Fix $\varepsilon > 0$.
Apply continuity of $f,h$ on $D$, and the fact that
$f' = h$ in the strong sense, to fix $\delta > 0$ such
that whenever $z - \delta < a_1 < a_2 < z$ with $a_1,a_2 \in D$,
the quantities $|s - \sigma|$, $|s_i - \sigma|$, 
$|b_i - f(z)|$, 
$|(f(a_2) - f(z))/ (a_2 - z) - \sigma|$ 
are all less than $\varepsilon$,
where $s_i = h(a_i)$ and $b_i = f(a_i)$, for $i=1,2$, and
$s = (b_2 - b_1)/(a_2 - a_1)$.
Let $M = \max(|s_1 - s|,|s_2 - s|)$, as in
Lemma \ref{lemma-interp-cubic};
so $M  \le 2 \varepsilon$.

Assume that $z$ is a limit
from the left of points of $D$ and of points of $\RRR \backslash D$;
otherwise checking continuity and the derivative
from the left is trivial.
Thus, $\delta$ may be taken small enough so that 
$(z-\delta,z)$ misses any unbounded interval in $\JJ$.
For
$a_1,a_2 \in D$ with $(a_1,a_2) \in \JJ$ and
$x\in \RRR$ with
$z-\delta < a_1 \le x < a_2 < z$, the bounds  
from Lemma \ref{lemma-interp-cubic} imply that
$|\tilde h (x) - \sigma| \le 
|\tilde h (x) - s | + |s - \sigma| \le 
3M + \varepsilon \le 7 \varepsilon $.
So $\tilde h$ is continuous.
To see that $h(z)=\tilde f'(z)$,
observe that by elementary geometry, 
the slope $(\tilde f(x) - \tilde f(z)) / (x - z)$ is between the
slopes $(\tilde f(x) - \tilde f(a_2)) / (x - a_2)$ and
$(\tilde f(a_2) - \tilde f(z)) / (a_2 - z)$.
Applying Lemma \ref{lemma-interp-cubic} again,
$|(\tilde f(x) - \tilde f(a_2)) / (x - a_2) - \sigma| \le
3M + \varepsilon \le 7 \varepsilon$,
so we are done.
\end{proof}

\section{Some Flavors of OCA}
\label{sec-oca}
The proofs of Theorems
\ref{thm-c1-pfa} and \ref{thm-c1} will require the results of this section.

\begin{definition}
For any set $E$,
let $E^\dag = (E \times E) \setminus \{(x,x) : x \in E\}$.
If $W \subseteq E^\dag$ with $W = W\iv$,  then
$T \subseteq E$ is \emph{$W$--free} iff
$T^\dag \cap W = \emptyset$,
and $T$ is \emph{$W$--connected} iff $T^\dag \subseteq W$.

Then SOCA is the assertion that whenever $E$ is an uncountable
separable metric space and $W = W\iv \subseteq E^\dag$ is open,
there is either
an uncountable $W$--free set or an uncountable $W$--connected set.
\end{definition}

SOCA follows from PFA, but not from $\MA(\aleph_1)$.
It clearly contradicts CH.
However, it is well-known \cite{Ga} that SOCA is
a ZFC theorem when $E$ is Polish:

\begin{lemma}
\label{lemma-polish-soca}
Assume that $E$ is an uncountable Polish space,
$W \subseteq E^\dag$ is open, and  $W = W\iv$.
Then there is a Cantor set $Q \subseteq E$ which is either
$W$--free or $W$--connected.
\end{lemma}
\begin{proof}
Shrinking $E$, we may assume that $E$ is a Cantor set;
in particular, non-empty open sets are uncountable.
Assume that no Cantor subset  is $W$--free.
Since $W$ is open, the closure of a $W$--free set is $W$--free;
thus every $W$--free set has countable closure, and is hence nowhere dense.

Now, inductively construct a tree, $\{P_s : s \in 2^{< \omega}\}$.
Each $P_s$ is a non-empty clopen subset of $E$,
with $\diam(P_s) \le 2^{-\lh(s)}$.
$P_{s\cat0}$ and $P_{s\cat1}$ are disjoint
subsets of $P_s$ such that
$(P_{s\cat0} \times P_{s\cat1}) \subseteq W$.
Let $Q = \bigcup\{\bigcap_n P_{f\res n} : f \in 2^\omega\}$;
then $Q$ is $W$--connected.
\end{proof}

An ``open covering'' version of SOCA follows by induction on $\ell$:

\begin{lemma}
\label{lemma-soca-cover}
Let $E$ be an uncountable separable metric space, with
$E^\dag = \bigcup_{i < \ell} W_i$, where $\ell \in \omega$
and each $W_i = W_i\iv$ is open in $E^\dag$. 
Assuming $\SOCA$, there is an uncountable $T \subseteq E$ such
that $T$ is $W_i$--connected for some $i$.  
In the case that $E$ is Polish, this is a $\ZFC$ result
and $T$ can be made perfect.
\end{lemma}

There is also a version of this lemma obtained by replacing
the covering by a continuous function:

\begin{lemma}
\label{lemma-cont-soca}
Assume that $E$ is an uncountable Polish space,
$F$ is a compact metric space, $g \in C(E^\dag, F)$, and
$g(x,y) = g(y,x)$ whenever $x \ne y$.
Then there is a Cantor set $Q \subseteq E$ such that
$g \res Q^\dag$ extends continuously to some $\hat g \in C(Q\times Q, F)$.
\end{lemma}
\begin{proof}
Construct a tree, $\{P_s : s \in 2^{< \omega}\}$.
Each $P_s$ is a Cantor subset of $E$,
with $\diam(P_s) \le 2^{-\lh(s)}$.
$P_{s\cat0}$ and $P_{s\cat1}$ are disjoint
subsets of $P_s$.
Also, apply Lemma \ref{lemma-soca-cover} to get
$\diam(g( P_s^\dag)) \le 2^{-\lh(s)}$.
Let $Q = \bigcup\{\bigcap_n P_{f\res n} : f \in 2^\omega\}$.
\end{proof}

Now, to prove Theorem \ref{thm-c1-pfa}, we need,
under PFA, a version of Lemma \ref{lemma-cont-soca}
where $E$ is just an uncountable subset of a Polish space.
We begin with the following, from
Abraham,  Rubin, and  Shelah \cite{ARS}:

\begin{theorem}
\label{thm-oca-ars}
Assume $\PFA$.  Then $\OCAARS$ holds.  That is,
let $E$ be a separable metric space of size $\aleph_1$.
Assume that $E^\dag = \bigcup_{i < \ell} W_i$, where $\ell \in \omega$
and each $W_i = W_i\iv$ is open in $E^\dag$.  Then $E$ can be partitioned
into sets $\{A_j : j \in \omega\}$ such that for each $j$,
$A_j$ is $W_i$--connected for some $i$.
\end{theorem}

The terminology $\OCAARS$ was used by Moore \cite{Mo} to distinguish
it from other flavors of the Open Coloring Axiom in the literature.
Actually, \cite{ARS} does not mention PFA, but rather its Theorem 3.1 shows,
by iterated ccc forcing, that
$\OCAARS$ is consistent with $\MA(\aleph_1)$;
but the same proof shows that it is true under PFA.
In our proof of Theorem \ref{thm-c1-pfa}, we only need
$\MA(\aleph_1)$ plus $\OCAARS$,
so in fact every model
of $2^{\aleph_0} = \aleph_1  \wedge 2^{\aleph_1} = \aleph_2$
has a ccc extension satisfying the result of Theorem \ref{thm-c1-pfa}.

To use $\OCAARS$ for our version of 
Lemma \ref{lemma-cont-soca}, we need the $A_j$
of Theorem \ref{thm-oca-ars} to be clopen.
This is not always possible, but can be achieved if we shrink $E$:

\begin{lemma}
\label{lemma-polish-cont}
Assume $\MA(\aleph_1)$.  Assume that $X$ is a Polish space
and $E \in [X]^{\aleph_1}$.  For each $n \in \omega$, let
$\{A^n_j : j \in \omega\}$ partition $E$ into $\aleph_0$ sets.
Then there is a Cantor set $Q \subseteq X$ and, for each $n$,
a partition of $Q$ into disjoint relatively clopen sets
$\{K^n_j : j \in \omega\}$  such that
$|Q \cap E| = \aleph_1$ and
each $K^n_j \cap  E = A^n_j \cap Q$.
\end{lemma}
\begin{proof}
Note that for each $n$, compactness of $Q$ implies that all but finitely
many of the $K^n_j$ will be empty.

For $s \in \omega^{< \omega}$, let $A_s = \bigcap\{A^n_{s(n)} : n < \lh(s)\}$,
with  $A_\emptyset = E$.
Shrinking $E,X$, we may assume that whenever $U \subseteq X$
is open and non-empty, $|E \cap U| = \aleph_1$
and each $|A_s \cap U|$ is either $0$ or $\aleph_1$.

Let $\BB$ be a countable open base for $X$, with $X \in \BB$.
Call $\TT$ a \emph{nice tree} iff:
\begin{itemizz}
\item[1.] $\TT$ is a non-empty subset of
$\BB \backslash \{\emptyset\}$ which is a tree under
the order $\subset$, with root node $X$.
\item[2.] $\TT$ has height $\hgt(\TT)$,
where $1 \le \hgt(\TT) \le \omega$. 
\item[3.] If $U \in \TT$ is at level $\ell$ with $\ell + 1  < \hgt(\TT)$, then
$U$ has finitely many but at least two children in $\TT$, and the closures
of the children are pairwise disjoint and contained in $U$.
\item[4.] If $U \in \TT$ is at level $\ell > 0$,
then $\diam(U) \le 1/\ell$.
\end{itemizz}
This labels the levels $0,1,2,\ldots$, with
$\hgt(\TT)$ the first empty level.
Let $L_\ell(\TT)$ be the set of nodes at level $\ell$.
By (1)--(3), each $L_\ell(\TT)$ is a finite pairwise disjoint collection.

When $\hgt(\TT) = \omega$, let $Q_\TT  =
\bigcap_{\ell\in \omega} \bigcup L_\ell(\TT) =
\bigcap_{\ell\in \omega} \cl( \bigcup L_\ell(\TT))$.
Then $Q_\TT$ is a Cantor set,
so it is natural to force with finite trees approximating $\TT$.
Since many Cantor sets are disjoint from $E$, each forcing condition
$p$ will have, as a side condition, a finite $I_p \subseteq E$
which is forced to be a subset of $Q$.

Define $p \in \PPP$ iff $p$ is a triple
$(\TT, I, \varphi) = (\TT_p, I_p, \varphi_p)$,
such that:
\begin{itemizz}
\item[a.] $\TT$ is a nice tree of some finite height $h = h_p \ge 1$.
\item[b.] $I$ is finite and $I \subseteq E \cap \bigcup L_{h-1}(\TT)$.
\item[c.] $\varphi : \TT \to \omega^{< \omega}$ with
$\varphi(U) \in \omega^\ell$ for $U \in L_\ell(\TT)$.
\item[d.]
$\varphi(V) \supseteq \varphi(U)$ whenever $V \subseteq U$.
\item[e.] If $s = \varphi(U)$ then $A_s \cap U \ne \emptyset$
and $I_p \subseteq A_s$.
\end{itemizz}

Define $q \le p$ iff $\TT_q$ is an end extension of $\TT_p$
and $I_q \supseteq I_p$ and $\varphi_q \supseteq \varphi_p$.
Then $\one = (\{X\}, \emptyset, \{(X, \emptyset)\} )$.
$\PPP$ is ccc (and $\sigma$--centered) because
$p,q$ are compatible whenever $\TT_p = \TT_q$ and $\varphi_p = \varphi_q$.
If $G$ is a filter meeting the dense sets $\{p : h_p > n\}$ for each $n$,
then $G$ defines a tree $\TT = \TT_G = \bigcup \{\TT_p : p \in G\}$
of height $\omega$, and $Q = Q_\TT$ is a Cantor set.
We also have $\varphi_G = \bigcup \{\varphi_p : p \in G\}$,
so $\varphi_G : \TT_G \to \omega^{< \omega}$;
also, let $I_G = \bigcup\{I_p : p \in G\}$.

Note that for each $x \in E$,
$\{p : x \in I_p \vee x \notin \bigcup L_{h_p-1}(\TT_p)\}$ is dense in $\PPP$.
If $G$ meets all these dense sets, then
$Q \cap E = I_G$.  We may then let $K^n_j = Q \cap
\bigcup\{U \in L_{n+1}(\TT_G) : \varphi(U)(n) = j\}$.

Finally, if we list $E$ as $\{e_\beta : \beta < \omega_1\}$,
note that each set $\{p : \exists \beta > \alpha \, [e_\beta \in I_p]\}$
is dense, so that we may force $Q \cap E$ to be uncountable.
\end{proof}

\begin{lemma}
\label{lemma-oca-cont}
Assume $\PFA$.  Assume that $X$ is a Polish space,
$F$ is a compact metric space, $E \in [X]^{\aleph_1}$,
$g \in C(X^\dag, F)$, and $g(x,y) = g(y,x)$ whenever $x \ne y$.
Then there is a Cantor set $Q \subseteq X$ such that
$|Q \cap E| = \aleph_1$ and $g \res Q^\dag$ extends continuously
to some $\hat g \in C(Q \times Q, F)$.
\end{lemma}
\begin{proof}
For each $n$, we may use compactness of $F$ to cover $X^\dag$ by
finitely many open sets, $W^n_i = (W^n_i)\iv$ for $i < \ell_n$,
such that each $\diam(g(W^n_i)) \le 2^{-n}$.
It follows by Theorem \ref{thm-oca-ars} that for each $n$,
we may partition $E$ into sets 
$\{A^n_j : j \in \omega\}$ such that each $A^n_j$ is
$W^n_i$--connected for some $i$,
so that $\diam(g((A_j^n)^\dag)) \le 2^{-n}$.

By Lemma \ref{lemma-polish-cont},  we have
a Cantor set $Q \subseteq X$ and, for each $n$,
a partition of $Q$ into disjoint relatively clopen sets
$\{K^n_j : j \in \omega\}$  such that
$|Q \cap E| = \aleph_1$ and
each $K^n_j \cap  E = A^n_j \cap Q$.
Shrinking $Q$, we may assume $Q \cap E$ is dense in $Q$,
so that each $A^n_j \cap Q$ is dense in $K^n_j$
and $\diam(g((K^n_j)^\dag)) \le 2^{-n}$.

Now, fix $x \in Q$.  For each $n$, $x$ lies in exactly one of the $K^n_j$,
and we may let $H^n= \cl(g( (K^n_j)^\dag)) $ for that $j$.
Then $\bigcap_n H^n$ is a singleton, and we may define
$\hat g$ on the diagonal by
$ \{ \hat g(x,x) \} = \bigcap_n H^n$.  It is easily seen that this $\hat g$ 
is continuous on $Q \times Q$.
\end{proof}

\section{Proofs of Positive Results}
\label{sec-posit}

\begin{lemma}
\label{lemma-dir}
Fix an uncountable $E \subseteq \RRR^n$  and an $\varepsilon > 0$.
Assuming $\SOCA$, there is an uncountable $T \subseteq E$ such
that $T$ is $\varepsilon$--directed.
In the case that $E$ is Polish, this is a $\ZFC$ result
and $T$ can be made perfect.
\end{lemma}
\begin{proof}
Let $\{V_i : i < \ell\}$ be an open cover of $S^{n-1}$ by sets
of diameter less than $\varepsilon$, and apply Lemma \ref{lemma-soca-cover}
with $W_i = \{ (x,y) \in E^\dag : \rho(x-y) \in V_i\}$.
\end{proof}

\begin{proofof}{Theorem \ref{thm-c1}}
Applying Lemma \ref{lemma-dir} and
shrinking $P$, we may assume that $P$ is
a Cantor set and that $P$ is $2 \sin(22.5^\circ)$--directed;
so, the direction between any
two points of $P$ is within $45^\circ$ of some fixed direction.
Rotating coordinates, we may assume that this fixed direction is along
the $x$-axis, where we label our $n$ axes as $x, y^1, \ldots, y^{n-1}$.
Now, $P$ is (the graph of) a function which expresses
$(y^1, \ldots, y^{n-1})$ as a function of $x$,
and $D := \dom(P)$ is a Cantor set.
Write $P(x)$ as $(P^1(x), \ldots, P^{n-1}(x))$.

The $xy^i$-planar slopes of $P$ are all in $[-1, 1]$.  That is,
for $x_1,x_2 \in D$ with $x_1 \ne x_2$, let
$g^i(x_1,x_2) = (P^i(x_2) - P^i(x_1)) / (x_2 - x_1)$;
then $|g^i(x_1, x_2)| \le 1$ for all $x_1,x_2$.
Each $g^i \in C(D^\dag, [0,1])$ and
$g^i(x_1,x_2) = g^i(x_2,x_1)$ whenever $x_1 \ne x_2$.
Applying Lemma \ref{lemma-cont-soca} with $F = [0,1]^{n-1}$ 
and shrinking $D$ if necessary, we may assume that each $g^i$
extends continuously to some $\hat g^i \in C(D\times D, [0,1])$.
Let $h^i(x) = \hat g^i (x,x)$.  Then $h^i$ is the derivative of $P^i$
in the strong sense.
Now, we may apply Lemma \ref{lemma-interp} on each coordinate
separately to obtain a $C^1$ arc $A \supseteq P$;  $A$ is 
the graph of a $C^1$ function $x \mapsto (A^1(x), \ldots, A^{n-1}(x))$
defined on an interval containing $D$.
\end{proofof}

\begin{proofof}{Theorem \ref{thm-c1-pfa}}
Given Lemma \ref{lemma-oca-cont}, the proof is almost identical
to the proof of Theorem \ref{thm-c1}.
\end{proofof}

When $E \subseteq \RRR^n$ has size exactly $\aleph_1$,
and the Question of Section \ref{sec-intro} has a positive answer,
it is natural to ask whether $E$ can
be covered by $\aleph_0$ ``nice'' arcs.
For example, under $\MA(\aleph_1)$, $E$ is covered
by $\aleph_0$ Cantor sets, and hence by $\aleph_0$ arcs.
One can also improve Theorem \ref{thm-c1-pfa}:

\begin{theorem}
PFA implies every $E \subseteq \RRR^n$ of size $\aleph_1$
can be covered  by $\aleph_0$ $\; C^1$  arcs.
\end{theorem}

The proof mimics the proof of Theorem \ref{thm-c1-pfa},
but uses improved versions of Lemmas \ref{lemma-dir},
\ref{lemma-polish-cont} and \ref{lemma-oca-cont}.
The new and improved Lemma \ref{lemma-dir} gets
$E$ covered by $\aleph_0$ $\varepsilon$--directed sets,
using Theorem \ref{thm-oca-ars} rather than SOCA.

The covering versions of Lemmas \ref{lemma-polish-cont}
and \ref{lemma-oca-cont} 
get Cantor sets $Q_\ell \subseteq X$ for $\ell \in \omega$ 
satisfying the conditions of the lemmas and
so that $E \subseteq \bigcup_\ell Q_\ell$.
To get the $Q_\ell$ for $\ell \in \omega$,
force with the finite support product of $\omega$
copies of the poset $\PPP$ described in the proof 
of Lemma \ref{lemma-polish-cont}.
Then, use the $Q_\ell$ to prove the covering version of
Lemma \ref{lemma-oca-cont}.
Even though the proof  of Lemma \ref{lemma-oca-cont}
shrinks $Q$, it does so by deleting at most countably many points from $E$,
so these points may be covered by $\aleph_0$ straight lines.
Thus, $E$ will be covered by $\bigcup_\ell Q_\ell$ together with
a countable union of lines.

\section{Proofs of Negative Results}
\label{sec-neg}

\begin{lemma}
\label{lemma-cantor-zero}
Let $D \subset \RRR$ be closed.
Then there is an $h \in C^\infty(\RRR)$ such that $h(x) \ge 0$ for all
$x$ and $D = \{x \in \RRR : h(x) = 0\}$.
\end{lemma}
\begin{proof}
Let $U = \RRR \backslash D$; we shall call our function $h_U$.
If $U = (a,b)$, then such $h_U$ are in standard texts;
for example,
let $h_{(a, b)} (x)$ be $\exp(-1 \div (x-a)(b-x) )$
for $x \in (a,b)$ and $0$ otherwise.
Now, say $U = \bigcup_{n \in \omega} J_n$, where each $J_n$
is a bounded open interval.  Let
$h_U = \sum_{n\in\omega} c_n h_{J_n}$, where each $c_n > 0$ and the
$c_n$ are small enough so that for each $\ell \in \omega$,
the $\ell^{\mathrm{th}}$ derivative $h^{(\ell)}_U$ is the uniform limit of
the sum $\sum_{n\in\omega} c_n h^{(\ell)}_{J_n}$. 
\end{proof}

\begin{proofof}{Theorem \ref{thm-squig}}
Let $D \subset \RRR$ be a Cantor set. 
Integrating the function of Lemma \ref{lemma-cantor-zero},
fix $f \in C^\infty(\RRR)$ such that $f'(x) \ge 0$ for all
$x$ and $D = \{x \in \RRR : f'(x) = 0\}$.  Then $f$ is strictly increasing.

Let $P$ be the graph of $f \res D$.  Fix an uncountable $A \subseteq P$,
and assume that $A$ is non-squiggly; we shall derive a contradiction.
Fix $\delta > 0$ as in Definition \ref{def-squig};
then, shrinking $A$, we may assume that $\diam(A) \le \delta$
so that whenever $\{x,y,z,t\} \in [A]^4$,
point $t$ is not interior to triangle $xyz$.

Let $S$ be an infinite subset of $\dom(A)$ such that every
point of $S$ is a limit, from the left and right, of other points of $S$.

Now, fix $a,b,c \in S$ with $a < b < c$; then $f(a) < f(b) < f(c)$.
Let $L$ be the straight line passing through $(a,f(a))$ and $(c,f(c))$.
Moving $b$ slightly if necessary, we may assume (since $f'(b) = 0$)
that $L$ does
not pass through $(b, f(b))$.  Then either $L(b) > f(b)$ or $L(b) < f(b)$.

Suppose that $L(b) > f(b)$.  Consider triangle $(a,f(a)), (b, f(b)), (c, f(c))$.
One leg of this triangle is the graph of $L \res [a,c]$, which passes
above the point $(b,f(b))$.  Since all three legs have positive slope
and $f'(b) = 0$, the points
$(b - \varepsilon, f(b - \varepsilon)) $
are interior to the triangle when $\varepsilon > 0$ is small enough.
Choosing such an $\varepsilon$ with $b - \varepsilon \in S$ yields 
a contradiction.

$L(b) < f(b)$ is likewise contradictory,
using points $(b + \varepsilon, f(b + \varepsilon)) $.
\end{proofof}

Observe that the arc in Theorem \ref{thm-squig}
cannot be real-analytic, since if $f : [0,1] \to \RRR$ is real-analytic,
then $[0,1]$ can be decomposed into finitely many intervals
on which either $f'' \ge 0$ or $f'' \le 0$.  On each of these
intervals, the graph of $f$ is non-squiggly.

\begin{proofof}{Theorem \ref{thm-c2}}
As in the proof of Theorem \ref{thm-squig}, 
let $D \subset \RRR$ be a Cantor set,
and fix $f \in C^\infty(\RRR)$ such that
$f$ is strictly increasing,
 $f'(y) \ge 0$ for all $y$,  and $D = \{y \in \RRR : f'(y) = 0\}$.
Also, to simplify notation, assume that $f(\RRR) = \RRR$, so that
$\varphi := f\iv \in C(\RRR)$ and is also a strictly increasing function.
Let $K = f(D)$; so $K$ is also a Cantor set.  Then $\varphi$ is $C^\infty$ on
$\RRR \backslash K$, and $\varphi'(x) = +\infty$ for $x \in K$.
Integrating, fix $\psi \in C^1(\RRR)$ such that $\psi' = \varphi$;
so $\psi$ is a convex function.

Note that whenever $x \in K$ and $M > 0$, there is an $\varepsilon > 0$
such that
$\varphi'(u) \ge M$ whenever $|u - x| < \varepsilon$.   When
$x - \varepsilon < a \le v \le b < x + \varepsilon$, we can integrate this to get
$\varphi(a) + M(v-a) \le \varphi(v) \le \varphi(b) - M(b-v)$.  Integrating again yields
\[
(b - a) \varphi(a) + (b - a)^2 M/2
\le \psi(b) - \psi(a) \le
(b - a) \varphi(b) - (b - a)^2 M/2 \ \ .
\]
This implies that, for $x \in K$,
\[
\lim_{t \to 0} \frac{ (\psi(x + t) - \psi(x)) / t - \varphi(x)} {  t}  = +\infty \  \ ;
\tag{$\ast$}
\]
the argument can be broken into
two cases:  $t \searrow 0$ (consider $a = x < x+t = b$) and
$t \nearrow 0$ (consider $a = x + t < x = b$).

Now let $P = \psi \res K$; so $P$ is a Cantor set in $\RRR^2$.
Suppose that $P$ meets the $C^2$ arc $A$ in an infinite set.
Since the intersection is compact, it contains a limit point $(x_0, y_0)$.
At $(x_0, y_0)$, the tangent to the arc $A$ is parallel to the tangent
of the $C^1$ arc $y = \psi(x)$; in particular, this tangent is not vertical.
Thus, replacing $A$ by a segment thereof, we may assume that 
$A$ is the arc $y = \xi(x)$, where $\xi$ is a $C^2$ function defined
in some neighborhood of $x_0$.
Now  $y_0= \xi(x_0) = \psi(x_0)$ and
$\xi'(x_0)= \psi'(x_0)  = \varphi(x_0)$.
Also, since $(x_0, y_0)$ is a limit point of the intersection,
there are non-zero $t_k$, for $k \in \omega$,
converging to $0$, such that each $\psi(x_0 + t_k) = \xi(x_0 + t_k)$.
Applying Taylor's Theorem to $\xi$,
\[
\psi(x_0 + t_k) = \psi(x_0) + \varphi(x_0) t_k + \frac{1}{2}\xi''(z_k)t_k^2  
\text{\ \ for some $z_k$ between } x_0 \text{ and } x_0 + t_k
 \  \  .
\]
Since $\xi''(z_k) \to \xi''(x_0)$, we have
\[
\big[ (\psi(x_0 + t_k) - \psi(x_0)) / t_k - \varphi(x_0)  \big] / t_k \to \xi''(x_0)/2 \ \ ,
\]
contradicting $(\ast)$.
\end{proofof}

If $\psi$ were $C^2$,
the limit in $(\ast)$ would be $\psi''(x)/2\ne\infty$
(by Taylor's Theorem).
Moreover, the Cantor set $P=\psi\res K$ 
meets \textit{any} $C^2$ arc in a finite set.
This illustrates a difference between $C^1$ and $C^2$:  rotation
can cure an infinite derivative, but not an infinite second derivative.
Even though $\varphi'(x)=\infty$ for $x\in K$,
rotating
the graph of $\varphi\res K$ gives us the graph of $f\res D$,
which lies on a $C^\infty$ arc.

\section{Remarks on Arcs}
\label{sec-rmks}

Although the notion of \emph{strongly} $C^k$ is the one capturing the 
geometric notion of ``smooth'',  
every polygonal path is weakly $C^\infty$. 
Moreover, the standard formulas
for evaluating line integrals (e.g.,
$\int_A \vec\Phi(\vec x) \cdot d\vec x =
\int_a^b \vec\Phi(\vec g(t)) \cdot \vec g\,'(t) \, dt$) only require 
the path $\vec g(t)$ to be \emph{weakly} $C^1$; the arc $A$ may have corners,
with the velocity vector $\vec g'(t)$ becoming zero at a corner. 

Theorems \ref{thm-c1-pfa}, \ref{thm-c1}, and \ref{thm-squig}
produce \emph{strongly} $C^k$ arcs.
In contrast, Theorem \ref{thm-c2} produces a perfect set 
which \emph{meets} all strongly $C^2$ arcs in a finite set.
Theorem \ref{thm-meet} shows that
the \textit{weakly} version of this theorem is false. 

To prove Theorem \ref{thm-meet},  we begin with an interpolation result.

\begin{definition}
An \emph{interpolation function} is a $\psi \in C([0,1], [0,1] )$
such that $\psi(0) = 0$ and $\psi(1) = 1$.
\end{definition}

\begin{definition}
Assume that $D$ is a closed subset of $[0,1]$ with $0,1 \in D$.
Fix $g \in C(D, \RRR^n)$, and
let $\psi$ be an interpolation function.
Then the $\psi$ \emph{interpolation for} $g$ is the function
$\tilde g \in  C([0,1], \RRR^n)$  extending $g$ such that
whenever $(a,b) $ is a maximal
interval in $[0,1]\backslash D$ and $u \in (a,b)$,
\[
\tilde g(u) = g(a) + (g(b) - g(a)) \psi( (u-a)/(b - a) ) \ \ .
\]
\end{definition}
It is easily seen that $\tilde g$ is indeed continuous on $[0,1]$.

\begin{definition}
Assume that $D$ is a closed subset of $[0,1]$ with $0,1 \in D$.
Then $g \in C(D, \RRR^n)$ is \emph{flat} iff 
for all $\alpha \in \omega$,
there is a bound $M_\alpha$ such that 
for all $u,t \in D$
$\| g(u) - g(t) \| \le M_\alpha |u - t|^\alpha$.
\end{definition}
That is, $g$ is flat iff
for all $\alpha \in \NNN=\omega\setminus\{0\}$,
$g$ is uniformly Lipschitz of order $\alpha$ on $D$.
If $D$ is finite, then every $g : D \to \RRR^n$ is flat.
If $D$ contains an interval, then a flat $g$ is constant on that interval,
because it is Lipschitz of order 2 there;  for $t < t + h$ in the interval: 
$\| g(t + h) - g(t) \| \le k \cdot M_2 \cdot h^2 / k^2$ for all $k \ge 1$.

\begin{lemma}
\label{lem-interpD}
Assume that $D$ is a closed subset of $[0,1]$ with $0,1 \in D$.
Assume that $g \in C(D, \RRR^n)$ is flat.
Let $\psi$ be an interpolation function such that $\psi \in C^\infty([0,1], [0,1] )$ and
$\psi^{(k)}(0) = \psi^{(k)}(1) = 0$ for all $k \in \NNN$.  
Let $\tilde g$ be the $\psi$ interpolation for $g$.
Then $\tilde g \in C^{\infty}([0,1], \RRR^n)$ and $\tilde g^{(k)}(t) = 0$ for all
$t \in D$ and all $k  \in \NNN$.
\end{lemma}
\begin{proof}
It is sufficient to produce bounds $B_{k}$ giving
the following Lipschitz condition for all
$t \in D$ and $u \notin D$:
\begin{itemizz}
\item[1.]
$\| \tilde g(u) -  \tilde g(t)\| \le B_{0} |u - t|^2 $ .
\item[2.]
$\| \tilde g^{(k)}(u) \| \le B_{k} |u - t|^2 $  for $k \in\NNN$.
\end{itemizz}
Note that (1)(2) fail for $u,t \notin D$, since the derivatives there
need not be $0$. On the other hand, (1) holds for $u,t \in D$,
because $g$ is flat.

Observe that (1) and 2-Lipschitz on $D$ 
prove $\tilde g'(t)=0$ for  $t\in D$, so that  
(2) makes $\tilde g \in C^1([0,1], \RRR^n)$.
For $k \ge 2$, induct on $k$ 
to see that $\tilde g \in C^{(k)}([0,1], \RRR^n)$:
(2) for $k-1$ and 
the fact that $\tilde g^{(k-1)}$ is 2-Lipschitz on $D$
prove $\tilde g^{(k)}(t)=0$ for $t \in D$, 
so (2) for $k$ makes $g^{(k)}$ continuous.

To prove  (1)(2), assume, without loss of generality, $t < u$.  
To handle (1)(2) together,
let $Q_0(u,t)= \| \tilde g(u) -  \tilde g(t)\|$, and for $k>0,$
$Q_k(u,t)=\| \tilde g^{(k)}(u) \|$. 
Consider the two cases:

Case I.  $(t,u) \cap D = \emptyset$: Say $t = a < u < b$,
where $a,b \in D$ and
$(a,b) $ is a maximal interval in $[0,1]\backslash D$.
So
\[
Q_k(u,t) = 
 \| g(b) - g(a)\|  \cdot
\left| \psi^{(k)}\left( \frac{u-a} {b - a} \right) \right| \cdot
\frac{1}{(b-a)^k } \ \ .
\]
Let $S_k$ be the largest value taken by the function $|\psi^{(k)}|$.
Consider:

Subcase I.1. $(b-a)^2 \le (u-a)$:  Here,
\[
Q_k(u,t) \le
 \|g(b) - g(a)\| \cdot S_k \cdot \frac{1}{(b-a)^{k } }\cdot
\frac{(u-a)^2}{(u-a)^2} \le  M_{k+4}S_k (u-a)^2 
 \ \ .
\]

Subcase I.2. $(b-a)^2 \ge (u-a)$:  In this case, 
use Taylor's Theorem and the assumption $\psi^{(n)}(0)=0$,
for all $n\in\NNN$, to bound
$|\psi^{(k)}(z)|$ by $\frac{S_{2k+4}}{(k+4)!}\,z^4$ .
Then,
\[
Q_k(u,t) \le 
M_0 \cdot \left| \psi^{(k)}\left( \frac{u-a} {b - a} \right) \right| \cdot 
\frac{(b-a)^{k+4}} {(u - a)^{k+4}} \cdot
\frac{(u-a)^{k+4}}{(b-a)^{2k+4}}  \le 
M_0 \cdot \frac{S_{2k+4}} {(k+4)!}\cdot (u-a)^2
   \  \ .
\]

Case II.  $(t,u) \cap D \ne \emptyset$:
Let $a = \sup(D \cap [t,u])$,
so $t < a < u$ and Case I applies to $a,u$. 
For (1), use the fact that $g$ is flat, together with
\[
\| \tilde g(u) -  \tilde g(t)\|  \le
\| \tilde g(u) -  \tilde g(a)\|  +
\| g(a) -  g(t)\|  \ \ .
\]
For (2), 
$\| \tilde g^{(k)}(u) \|   \le  B_k |u - a|^2 \le B_k |u-t|^2$.
\end{proof}

\begin{proofof}{Theorem \ref{thm-meet}}
Passing to a subset, and possibly translating it, let
$E = \{\vec x_j: j \in \omega\}$, where the $\vec x_j$
converge to $\vec 0$, and 
\begin{itemizz}
\item[a.] $\| \vec x_0 \| >   \| \vec x_1 \| >   \| \vec x_2 \| >   \cdots $.
\item[b.] $\| \vec x_j \| \le 2^{- j ^2 }$ for each $j$.
\end{itemizz}
Let $A$ be the set obtained by connecting
each $\vec x_j$ to $\vec x_{j+1}$ by a straight line segment;
so $A$ is a ``polygonal'' arc,
with $\omega$ steps. 
Moreover, the natural path which traverses it
from $\vec 0$ to $\vec x_0$ will be 1-1,
because (a) guarantees that the line segments forming $A$ meet
only at the $\vec x_j$.
Let $D = \{0\} \cup \{2^{-j}: j \in \omega\}$, and
define $g : D \to \RRR^n$ by $g(0) = \vec 0$ and
$g(2^{-j}) = \vec x_j$.  Then $g$ is flat, by (b)
(with $M_\alpha = 2^{1+\alpha + \alpha^2}$).  

Let $\psi\in C^\infty(\RRR)$ be such that 
\begin{itemizn}{43}
\item
$\psi(t) = 0$ when $t \le 0$ and $\psi(t) = 1$ when $t \ge 1$.
\item
$\psi'(t) > 0$ for $0 < t < 1$.
\item
$\psi^{(k)}(0)=\psi^{(k)}(1) = 0$ for $k \ge 1$.
\end{itemizn}
Such a $\psi$ may be obtained by
integrating a scalar multiple of the function described in 
Lemma \ref{lemma-cantor-zero}.
Let $\tilde g:  [0,1] \to \RRR^n$
be the $\psi$ interpolation for $g$.
Then, by Lemma \ref{lem-interpD}, $\tilde g\in C^\infty([0,1],\RRR^n)$.
\end{proofof}

For the path $\tilde g$ in the preceding proof,
all $\tilde g^{(k)}$ (for $k \ge 1$)
will be $\vec 0$ when passing through each $\vec x_j$, so that no
acceleration is felt when rounding a corner.  Also, each $\tilde g^{(k)}$
will be $\vec 0$ at $t = 0$.

Now consider the perfect set version. 

\begin{theorem}
If $E \subseteq \RRR^n$ is Borel and uncountable, then
$E$ meets some weakly $C^\infty$ arc in an uncountable set.
\end{theorem}
\begin{proof}
Write elements of $\RRR^n$ as $\vec x = (x^1, \ldots, x^n)$.
By shrinking and rotating $E$, we may assume that $E$ is a Cantor set
and the projection $\pi^1$ of $E$ on the $x^1$ coordinate is 1-1.
Shrinking $E$ further, we may assume that 
$E = \bigcap_j ( \bigcup \{F_\sigma : \sigma \in \{0,2\}^j  \} )$,
where the $F_\sigma$ are compact and form a tree and each
$\diam(F_\sigma) \le 3^{- (\lh(\sigma))^2}$.

In $\RRR$, the ``$t$--axis'', let $D$ be the usual middle-third Cantor set.
Then $D = \bigcap_j ( \bigcup \{I_\sigma : \sigma \in \{0,2\}^j  \} )$, where 
$I_\sigma$ is an interval of length $3^{- \lh(\sigma)}$.
Let $g: D \onto E$ be the natural homeomorphism. 
So, if $\alpha \in  \{0,2\}^\omega$,
it determines the point
$t_\alpha = \sum_{i \in \omega} (\alpha_i 3^{-i}) \in D$.
Then $\bigcap_{i \in \omega} I_{\alpha \res i} = \{t_\alpha\}$ and
$\bigcap_{i \in \omega} F_{\alpha \res i} = \{g(t_\alpha)\}$.

Note that $g$ is flat.  Let $\psi\in C^\infty(\RRR)$
be as in the proof of Theorem \ref{thm-meet}, and let $\tilde g$
be the $\psi$ interpolation for $g$.  
Then $\tilde g \in C^\infty([0,1],\RRR^n)$.

Finally, in choosing $E$ and the $F_\sigma$, make sure that
if $\sigma < \tau$ lexicographically, then all elements
of $\pi^1(F_\sigma)$ are less than all elements of $\pi^1(F_\tau)$.
This will guarantee that $\pi^1 \circ g: D \to \RRR$ is
order-preserving, so that $\tilde g$ is a 1-1 function.
\end{proof}

Under $\MA(\aleph_1)$, if
$E \subseteq \RRR^n$ has size $\aleph_1$, then $E$ can be covered
by $\aleph_0$ weakly $C^\infty$ arcs.
In particular, $E$ can be covered by
$\aleph_0$ copies, or rotated copies, of the perfect set $g(D)$
constructed in the preceding proof.

\renewcommand\refname

\end{document}